\documentclass[12pt]{amsart}

\usepackage{geometry}
\usepackage[normalem]{ulem}
\usepackage{tikz}
\usepackage{hyperref} 
\usepackage{textcase}
\usepackage{longtable}
\usepackage{adjustbox}
\usepackage{float}
\usepackage{tabularx,hhline,multirow}
\usepackage{mathrsfs, amscd,amssymb,amsthm,amsmath,csquotes,graphicx,array,wasysym, textcomp,comment,epigraph,rotating,fancyvrb,pdflscape,upgreek}
\usepackage[matrix,arrow,curve]{xy}
\usepackage{graphicx}
\usepackage{tikz-cd}
\usepackage{makecell}
\usepackage{ragged2e}

\newcommand{\numl}{\operatorname{\#}}
\newcommand{\FF}{\mathbb{F}}
\newcommand{\BB}{\mathbb{B}}

\newcommand{\PP}{\mathbb{P}}

\newcommand{\ZZ}{\mathbb{Z}}
\newcommand{\GG}{\mathbb{G}}
\newcommand{\UUU}{\mathbb{U}}
\newcommand{\KK}{\Bbbk}

\newcommand{\DD}{\mathrm{D}}
\newcommand{\mumu}{\boldsymbol{\mu}}

\newcommand{\Bl}{\operatorname{Bl}}

\newcommand{\Aut}{\operatorname{Aut}}
\newcommand{\GL}{\operatorname{GL}}
\newcommand{\PGL}{\operatorname{PGL}}
\newcommand{\Bir}{\operatorname{Bir}}
\newcommand{\Type}{\operatorname{Type}}
\newcommand{\ind}{\tau}
\newcommand{\Ga}{\mathbb{G}_{\mathrm{a}}}
\newcommand{\Gm}{\mathbb{G}_{\mathrm{m}}}
\newcommand{\UU}{\mathbb{U}}
\newcommand{\type}[1]{\mathrm{#1}}

\newtheorem{theorem}[subsection]{Theorem}
\newtheorem{lemma}[subsection]{Lemma}

\newtheorem{proposition}[subsection]{Proposition}
\newtheorem{claim}[subsection]{Claim}
\theoremstyle{definition}
\newtheorem{definition}[subsection]{Definition}
\newtheorem{remark}[subsection]{Remark}

\newtheorem{notation}[subsection]{Notation}

\newcounter{No}
\renewcommand{\theNo}{{{\rm\arabic{No}$^o$}}}
\def\no{\refstepcounter{No}{\theNo}}

\title{Automorphisms of Du Val del Pezzo surfaces}
\author{Nikita Virin}
\address{Lomonosov Moscow State University, Moscow, Russia}
\email{virinnikita@gmail.com}
 
\begin{document}

\begin{abstract}
    We present a description for the automorphism groups of Du Val del Pezzo surfaces whose automorphism groups are infinite.
\end{abstract}

\keywords{
    Du Val singularity, del Pezzo surface, automorphism.
}

\subjclass{Primary 14J50; Secondary 14J17, 14J26}

\maketitle

\section{Introduction}

We assume that all varieties are projective and defined over an algebraically closed field $\KK$ of characteristics $0$.

A \textit{Du Val del Pezzo surface} is a normal projective surface with at worst Du Val singularities and ample anticanonical divisor. Automorphism groups of smooth del Pezzo surfaces are well studied (see, for example, \cite{dol}). In \cite{che} and \cite{mar}, Du Val del Pezzo surfaces with infinite automorphism groups have been classified and connected components $\Aut^0(X)\subset \Aut(X)$ of their automorphism groups have been described. In this paper we describe \textit{groups of connected components} $\Aut(X)/\Aut^0(X)$ for all such Du Val del Pezzo surfaces. Moreover, it is shown that in most cases the group $\Aut(X)$ is a semidirect product of its connected component and a finite group. More precisely, we prove the following theorem.

\begin{theorem}
    Let $X$ be a Du Val del Pezzo surface with infinite automorphism group $\Aut(X)$. Then $\Aut(X)$ is described by the column $\Aut(X)$ in the table from Appendix \textup{\ref{section:tables}}. Moreover, except for the cases \ref{d=1:2D4}, \ref{d=2:2A3}, \ref{d=3:2A2} and \ref{d=4:2A1-8lines}, there is a decomposition
    \[
        \Aut(X) = \Aut^0(X)\rtimes (\Aut(X)/\Aut^0(X)).
    \]
\end{theorem}

Note that the description of $\Aut(X)$ for Du Val del Pezzo surfaces of degree $3$ that have no parameters was obtained in \cite{sak}.

The author is grateful to Yuri Prokhorov for stating the problem, useful discussions and constant attention to this work. Also the author would like to thank Andrey Trepalin and Constantin Shramov for useful discussions.

\section{Preliminaries}

We use the following notation.

\begin{itemize}
\item
$\Ga$ is a one-dimensional unipotent additive group.
\item
$\Gm$ is a one-dimensional algebraic torus.
\item
$\BB_n$ is a Borel subgroup of $\PGL_n(\KK)$.
\item
$\UUU_n$ is a maximal unipotent subgroup of $\PGL_n(\KK)$.
\item
$\Ga\rtimes_{(n)} \Gm$ is a semidirect product of $\Ga$ and $\Gm$ such that $\Gm$ acts on $\Ga$ as $a \mapsto t^n a$.
\item
$G_1\rtimes_\varphi G_2$ is a semidirect product of $G_1$ and $G_2$ such that $G_2$ acts on $G_1$ as $g_1 \mapsto \varphi_{g_2}(g_1)$, where $\varphi: G_2\to \Aut(G_1)$ is a group homomorphism and $g_1\in G_1$, $g_2\in G_2$. We denote $\varphi_{g_2}(g_1)$ by $g_2\cdot_\varphi g_1.$ We also denote $\varphi_{g_2}(g_1)$ by $g_2\cdot g_1,$ where $\varphi$ is a conjugation action.
\end{itemize}

\begin{notation}
    Let $G$ be an algebraic group, let $g\in G$. Then $g_s$ denotes the semisimple part of $g$, and $g_u$ denotes the unipotent part of $g$. Notice that $g=g_ug_s$.
\end{notation}

\begin{definition}
    A quasitorus is an algebraic group $T$ isomorphic to $\Gm^n\times A$, where $A$ is a finite abelian group.
\end{definition}

The following lemma is a basic fact about semisimple elements of algebraic groups.

\begin{lemma}[{\cite[\S 16]{hum}}]
    \label{l:t}
    Let $G$ be a linear algebraic group. Let $g\in G$ be a semisimple element. Then there is a quasitorus $T \subset G$ such that $g\in T$.
\end{lemma}

This lemma helps us to prove Lemma \ref{l1} and Lemma \ref{l2}.

\begin{proposition}[see e.g. {\cite[Lemma 4]{pop}}]
    \label{cl:r}
    Let $X$ be an irreducible variety whose automorphism group $\Aut(X)$ is a linear algebraic. Let $G$ be a reductive subgroup of $\Aut(X)$, and let $P \in X$ be a fixed point of $G$. Then the natural representation $$G \longrightarrow \GL(T_{P,X})$$ is faithful.
\end{proposition}

    The following proposition is a particular case of a well-known fact.
    
\begin{proposition}[{see, for example, \cite[Corollary 3.1.3]{kuz}}]
    Let $X$ be a Du Val del Pezzo surface. Then $\Aut(X)$ is a linear algebraic group.
\end{proposition}

\begin{lemma}[{\cite[Lemma 8.1.11]{dol}}]
\label{lg1}
Let $X$ be a Du Val del Pezzo surface and let $\widetilde{X}\to X$ be its minimal resolution. Then any irreducible reduced curve on $\widetilde{X}$ with negative self-intersection is either a $(-1)$-curve or $(-2)$-curve.
\end{lemma}

\begin{lemma}[{\cite[\S 8]{dol}}]
    \label{lg2}
    Let $X$ be a Du Val del Pezzo surface and let $\widetilde{X}\to X$ be its minimal resolution. Then $\widetilde{X}$ has a finite number of irreducible reduced curves with negative self-intersection.
\end{lemma}

Now we give the following definition.

\begin{definition}
    A \textit{dual graph} of a Du Val del Pezzo surface $X$ is the graph whose vertices are the curves with negative self-intersection numbers on the minimal resolution of singularities $\widetilde{X}$. The number of edges between two distinct vertices $E_i$ and $E_j$ equals $E_i \cdot E_j$. We denote a $(-1)$-curve by \begin{tikzpicture}
        \node[circle,fill=black] (1) at (1,0) {};
    \end{tikzpicture}, we denote a $(-2)$-curve by \begin{tikzpicture}
        \node[circle,draw=black] (1) at (1,0) {};
    \end{tikzpicture}, and we denote the dual graph by $\Gamma(X)$.
\end{definition}

Let $X$ be a Du Val del Pezzo surface. Clearly, $\Aut(X)$ acts on $\Gamma(X)$. In the next section, we discuss some properties of this action.

\section{Actions of automorphism groups on dual graphs}

In the sequel, $X$ denotes a Du Val del Pezzo surface and $\mu:\widetilde{X}\to X$ denotes its minimal resolution.

\begin{lemma}
    \label{l1}
    Let $G\subset \Aut(X)$ be an algebraic subgroup. Let $E$, $E_1$, $E_2$, $E_3$ be distinct $G$-invariant smooth rational curves on $\widetilde{X}$ such that $E\cdot E_i=1$ and $E_i\cdot E_j=0$ for $i\neq j$. In addition, suppose that $G$ contains a one-dimensional torus $T$. Then $G$ is connected.
\end{lemma}

\begin{proof}
    Let $P$ be the intersection point of $E$ and $E_1$. Let $v$ and $v_1$ be nonzero tangent vectors to $E$ and $E_1$ at $P$, respectively. Since $E$ and $E_1$ intersect transversally, we see that $v$ and $v_1$ form a basis for $T_{P,X}$.
    Obviously, the group action of $G$ on $E$ is trivial. Then $v$ is $G$-invariant and $v_1$ is a $G$-eigenvector. 
    Consider the natural representation $$\varphi : G \longrightarrow \GL(T_{P,X}).$$ 
    By the above $\varphi(G)\subset T'$, where $T'$ is a one-dimensional torus in $\GL(T_{P,X})$. From Proposition \ref{cl:r} it follows that $\varphi(T)=T'$. Hence $\varphi(G)=T'$. 
    
    Let us show that $G$ is connected. Let $g\in G$. Then there is $t\in T$ such that $\varphi(g)=\varphi(t)$. Since the element $(tg^{-1})_u^{-1} tg^{-1}$ is semisimple, we see that this element lies in a quasitorus by Lemma \ref{l:t}. Natural representation of this quasitorus is faithful by Proposition \ref{cl:r}. Clearly, $\varphi(tg^{-1})=1$. Therefore, we have $$\varphi((tg^{-1})_u^{-1} tg^{-1}) = \varphi(tg^{-1})_u^{-1}\varphi(tg^{-1}) = 1.$$ Thus, since $t\in T\subset G^0$ and $(tg^{-1})_u\in G^0$, $g=(tg^{-1})_u^{-1} t \in G^0$. This completes the proof of the lemma.
\end{proof}

The following lemma can be proved similarly.

\begin{lemma}
    \label{l2}
    Let $G\subset \Aut(X)$ be an algebraic subgroup. Let $E_1$, $E_2$ be distinct $G$-invariant smooth rational curves on $\widetilde{X}$ such that $E_1\cdot E_2=1$. In addition, suppose that $G$ contains a two-dimensional torus $T$. Then $G$ is connected.
\end{lemma}

Using these lemmas we are going to prove the following proposition.

\begin{proposition}
    Let $X$ be a surface of one of the types \ref{d=1:E8}-\ref{d=7} except for the cases \ref{d=2:A7}, \ref{d=3:A5}, and \ref{d=4:A3-4lines}. Then there is an exact sequence $$1 \longrightarrow \Aut^0(X) \longrightarrow \Aut(X) \longrightarrow \Aut(\Gamma(X)).$$
\end{proposition}

\begin{proof}
    Let $X$ be such surface. Let $\varphi \in \Aut(X)$ such that $\varphi$ acts trivially on $\Gamma(X)$. We only need to prove that $\varphi\in \Aut^0(X)$. By the classification of Du Val del Pezzo surfaces whose automorphism groups are infinite, $\Aut^0(X)$ contains a torus of positive dimenseion. We want to apply either Lemma \ref{l1} or Lemma \ref{l2} to the algebraic subgroup $G\subset\Aut(X)$ generated by $\varphi$ and $\Aut^0(X)$. Notice that $G=\langle \varphi\rangle\cdot \Aut^0(X)$. We need to find curves as in the lemmas. Since every curve with negative self-intersection number on $\widetilde{X}$ is rational, smooth, and $G$-invariant, we see that the structure of the dual graph shows (see \cite[Appendix A]{che}) that such curves exist. Thus, $G=\Aut^0(X)$ and $\varphi \in \Aut^0(X)$.
\end{proof}

The following proposition is well-known (see e.g. \cite{dol}).

\begin{proposition}
    \label{c4}
    \begin{itemize}
    \item 
    $\Aut(X)=\Gm^2\rtimes \DD_6$, where $X$ is the smooth del Pezzo surface of degree $6$,
    \item
    $\Aut(\Bl_P(\PP^1\times \PP^1))= (\BB_2\times \BB_2)\rtimes \ZZ/2\ZZ$,
    \item
    $\Aut (\PP(1,1,2))=\GG_a^3\rtimes(\GL_2(\KK)/\mumu_2)$,
    \item
    $\Aut (\FF_1)=\GG_a^2\rtimes\GL_2(\KK)$,
    \item
    $\Aut(\PP^1\times \PP^1)=(\PGL_2 (\KK)\times \PGL_2 (\KK))\rtimes \ZZ/2\ZZ$,
    \item
    $\Aut(\PP^2)=\PGL_3 (\KK)$. 
    \end{itemize}
\end{proposition}

The above proposition gives a description of the automorphism group in the cases  \ref{d=6:smooth}, \ref{d=7:smooth}, \ref{d=8}, \ref{d=8:F1}, \ref{d=8:P1-P1}, and \ref{d=9:P2}, respectively.

\begin{lemma}
    \label{lm}
    Let $X$ be the surface of type \ref{d=4:4A1}, \ref{d=4:3A1}, \ref{d=5:A2}, \ref{d=5:2A1}, \ref{d=5:A1}, \ref{d=6:A2}, \ref{d=6:A1-3l} or \ref{d=6:A1-4l}. Then the natural homomorphism $$\Aut(X) \longrightarrow \Aut(\Gamma(X))$$ is surjective and has a section.
\end{lemma}

\begin{proof}
Consider cases \ref{d=4:4A1}, \ref{d=4:3A1}, \ref{d=5:A2}, \ref{d=5:2A1}, \ref{d=5:A1}, \ref{d=6:A2}, \ref{d=6:A1-3l}, \ref{d=6:A1-4l} separately.
\subsection*{Case \ref{d=4:4A1}}
 Let $X$ be the surface of type \ref{d=4:4A1}. The dual graph of $X$ is as follows
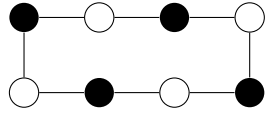
\begin{center}
\begin{tikzpicture}
\node[circle,draw=black] (0) at (1,-1) {};
	\node[circle,fill=black] (1) at (1,0) {};
	\node[circle,draw=black] (2) at (2,0) {};
	\node[circle,fill=black] (3) at (3,0) {};
	\node[circle,draw=black] (4) at (4,0) {};
	\node[circle,fill=black] (5) at (4,-1) {};
\node[circle,draw=black] (6) at (3,-1) {};
\node[circle,fill=black] (7) at (2,-1) {};
\path (0) edge (1);
	\path (1) edge (2);
	\path (2) edge (3);
	\path (3) edge (4);
	\path (4) edge (5);
\path (5) edge (6);
\path (6) edge (7);
\path (7) edge (0);
\end{tikzpicture}
\end{center}
and $\Aut(\Gamma(X))=\DD_4$.
Consider $\PP^1\times\PP^1$ with the points $P=(0,0)$, $P'=(0,1)$, $Q=(1,0)$, and $Q'=(1,1)$. Then there exists a group $G\subset \Aut(\PP^1\times\PP^1)$ such that $G$ preserves the set $\{P, P', Q, Q'\}$, an action on the set $\{P, P', Q, Q'\}$ is faithful, and $G$ is isomorphic to $\DD_4$. Obviously, $G$ lifts to the subgroup $\widetilde{G}\subset\Aut(\Bl_{P, P', Q, Q'}(\PP^1\times\PP^1))$. Notice that $\widetilde{X}$ is isomorphic to $\Bl_{P, P', Q, Q'}(\PP^1\times\PP^1)$. We see that the natural homomorphism $\widetilde{G}\to\Aut(\Gamma(X))$ is an isomorphism, so we are done.

Similarly, if $X$ is of type \ref{d=5:2A1} or \ref{d=6:A1-4l}, then the natural homomorphism $\Aut(X) \to \Aut(\Gamma(X))$ is surjective and has a section.

\subsection*{Case \ref{d=4:3A1}}
Let $X$ be the surface of type \ref{d=4:3A1}. The dual graph of $X$ is as follows
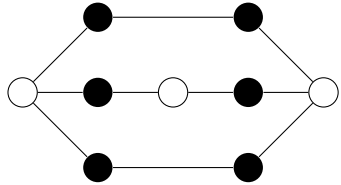
\begin{center}
\begin{tikzpicture}
\node[circle,fill=black] (0) at (0,0) {};
	\node[circle,draw=black] (1) at (1,0) {};
	\node[circle,draw=black] (2) at (-1,0) {};
	\node[circle,fill=black] (3) at (0,1) {};
	\node[circle,fill=black] (4) at (0,-1) {};
\node[circle,fill=black] (5) at (2,1) {};
\node[circle,draw=black] (6) at (3,0) {};
\node[circle,fill=black] (7) at (2,-1) {};
\node[circle,fill=black] (8) at (2,0) {};
	\path (5) edge (6);
	\path (2) edge (3);
	\path (2) edge (4);
\path (6) edge (7);
\path (3) edge (5);
\path (4) edge (7);
\path (8) edge (6);
\path (8) edge (1);
\path (0) edge (1);
\path (0) edge (2);
\end{tikzpicture}
\end{center}
and $\Aut(\Gamma(X))=(\ZZ/2\ZZ)^2$.
Consider $\PP^2$ with the points $P=(1:0:0)$, $Q=(0:1:0)$, $R=(0:0:1)$. Let $L'=\{x=0\}$ and $L=\{x+y=0\}$. The transformations 
    \begin{eqnarray*}
    (x:y:z)&\longmapsto& (y:x:z)
    \\
    (x:y:z)&\longmapsto& (yz:xz:xy)
    \end{eqnarray*}
    generate a subgroup $G\subset\Bir(\PP^2)$. This subgroup is isomorphic to $(\ZZ/2\ZZ)^2$. Now, consider $\varphi: \Bl_{P,Q,R}(\PP^2)\to\PP^2$. Notice that $\Bl_{P,Q,R}(\PP^2)$ is isomorphic to $\widetilde{Y}$, where $Y$ is the surface \ref{d=6:smooth}. Obviously, $G$ lifts to the subgroup $\widetilde{G}\subset\Aut(\widetilde{Y})$. 
     Consider $P'=\varphi^{-1}(P)\cap \widetilde{L}$ and $P''=\widetilde{L'}\cap\widetilde{L}$. Notice that $\Bl_{P', P''}(\widetilde{Y})$ is isomorphic to $\widetilde{X}$. Obviously, $\widetilde{G}$ lifts to the subgroup $\widehat{G}\subset\Aut(\widetilde{X})$. Now, we see that the natural homomorphism $\widehat{G}\to\Aut(\Gamma(X))$ is an isomorphism, so we are done.

\subsection*{Case \ref{d=5:A1}}
Let $X$ be the surface of type \ref{d=5:A1}. Then $\Aut^0(X)$ is isomorphic to $\Gm$. The dual graph of $X$ is as follows
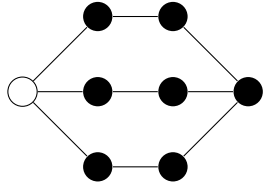
\begin{center}
\begin{tikzpicture}
\node[circle,fill=black] (0) at (0,0) {};
	\node[circle,fill=black] (1) at (1,0) {};
	\node[circle,draw=black] (2) at (-1,0) {};
	\node[circle,fill=black] (3) at (0,1) {};
	\node[circle,fill=black] (4) at (0,-1) {};
\node[circle,fill=black] (5) at (1,1) {};
\node[circle,fill=black] (6) at (2,0) {};
\node[circle,fill=black] (7) at (1,-1) {};
	\path (5) edge (6);
	\path (2) edge (3);
	\path (2) edge (4);
\path (6) edge (7);
\path (3) edge (5);
\path (4) edge (7);
\path (1) edge (6);
\path (0) edge (1);
\path (0) edge (2);
\end{tikzpicture}
\end{center}
and $\Aut(\Gamma(X))=\DD_3$.
Now, consider $\FF_1$ with a smooth rational curve $E\subset\FF_1$ with self-intersection $1$. Consider distinct points $P, Q, R\in E$. Then there exists a group $G\subset \Aut(\FF_1)$ such that $G$ preserves the set $\{P, Q, R\}$, an action on the set $\{P, Q, R\}$ is faithful, and $G$ is isomorphic to $\DD_3$.  Obviously, $G$ lifts to the subgroup $\widetilde{G}\subset\Aut(\Bl_{P, Q, R}(\FF_1))$. Notice that $\Bl_{P, Q, R}(\FF_1)$ is isomorphic to $\widetilde{X}$. Therefore, we see that the natural homomorphism $\widetilde{G}\to\Aut(\Gamma(X))$ is an isomorphism, so we are done. Moreover, $\Aut(X)$ is isomorphic to $\Aut^0(X)\times\Aut(\Gamma(X))$.

\subsection*{Case \ref{d=6:A2}}
Let $X$ be the surface of type \ref{d=6:A2}. Then $\Aut^0(X)$ is isomorphic to $\UUU_3\rtimes\Gm$. The dual graph of $X$ is as follows
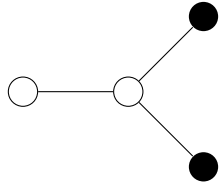
\begin{center}
\begin{tikzpicture}
	\node[circle,draw=black] (1) at (0,0) {};
	\node[circle,draw=black] (2) at (-1.4,0) {};
\node[circle,fill=black] (3) at (1,1) {};
\node[circle,fill=black] (4) at (1,-1) {};
	\path (1) edge (2);
\path (1) edge (3);
\path (1) edge (4);
\end{tikzpicture}
\end{center}
and $\Aut(\Gamma(X))=\ZZ/2\ZZ$.
Now, consider $\FF_2$ with a fiber $F\subset\FF_2$. Then consider distinct points $P, Q\in F$. There exists a group $G\subset \Aut(\FF_2)$ such that $G$ preserves the set $\{P, Q\}$, an action on the set $\{P, Q\}$ is faithful, and $G$ is isomorphic to $\ZZ/2\ZZ$.  Obviously, $G$ lifts to the subgroup $\widetilde{G}\subset\Aut(\Bl_{P, Q}(\FF_2))$. Notice that $\Bl_{P, Q}(\FF_2)$ is isomorphic to $\widetilde{X}$. We see that the natural homomorphism $\widetilde{G}\to\Aut(\Gamma(X))$ is an isomorphism, so we are done.

Similarly, if $X$ is of type \ref{d=5:A2}, then $\Aut(X) \to \Aut(\Gamma(X))$ is surjective and has a section.

\subsection*{Case \ref{d=6:A1-3l}}
Let $X$ be the surface of type \ref{d=6:A1-3l}. Then $\Aut^0(X)$ is isomorphic to $\GG_a^2\rtimes\Gm$. The dual graph of $X$ is as follows
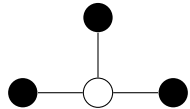
\begin{center}
\begin{tikzpicture}
	\node[circle,draw=black] (1) at (0,0) {};
	\node[circle,fill=black] (2) at (-1,0) {};
\node[circle,fill=black] (3) at (0,1) {};
\node[circle,fill=black] (4) at (1,0) {};
	\path (1) edge (2);
\path (1) edge (3);
\path (1) edge (4);
\end{tikzpicture}
\end{center}
and  $\Aut(\Gamma(X))=\DD_3$.
Consider $\PP^2$ with a line $L\subset\PP^2$. Then consider distinct points $P, Q, R\in L$. There exists a group $G\subset \Aut(\PP^2)$ such that $G$ preserves the set $\{P, Q, R\}$, an action on the set $\{P, Q, R\}$ is faithful, and $G$ is isomorphic to $\DD_3$.  Then $G$ lifts to the subgroup $\widetilde{G}$ of $\Aut(\Bl_{P, Q, R}(\PP^2))$. Notice that $\Bl_{P, Q, R}(\PP^2)$ is isomorphic to $\widetilde{X}$. We see that the natural homomorphism $\widetilde{G}\to\Aut(\Gamma(X))$ is an isomorphism, so we are done.
\end{proof}

Now, we are ready to prove the following proposition.

\begin{proposition}
    \label{cm}
    Let $X$ be the surface of type \ref{d=1:E8}-\ref{d=6:A1-4l} or \ref{d=7} except for the cases \ref{d=1:2D4}, \ref{d=2:A7}, \ref{d=2:2A3}, \ref{d=3:A5}, \ref{d=3:2A2}, \ref{d=4:A3-4lines} and \ref{d=4:2A1-8lines}. Then the natural homomorphism $$\Aut(X) \longrightarrow \Aut(\Gamma(X))$$ is surjective and has a section.
\end{proposition}

\begin{proof}
    In the cases \ref{d=1:E8}, \ref{d=1:E7-A1}, \ref{d=2:E7}, \ref{d=2:D6-A1}, \ref{d=3:E6}, \ref{d=3:A5-A1}, \ref{d=3:D5}, \ref{d=3:A4-A1}, \ref{d=4:D5}, \ref{d=4:A4}, \ref{d=5:A4}, \ref{d=5:A3}, \ref{d=5:A2-A1}, \ref{d=6:A2-A1}, \ref{d=6:2A1}, \ref{d=7} the group $\Aut(\Gamma(X))$ is trivial, so we are done.

    Let $X$ be the surface of type \ref{d=4:4A1}, \ref{d=4:3A1}, \ref{d=5:A2}, \ref{d=5:2A1}, \ref{d=5:A1}, \ref{d=6:A2}, \ref{d=6:A1-3l}, \ref{d=6:A1-4l}. Then the natural homomorphism is surjective and has a section by Lemma \ref{lm}.

    Let $X$ be the surface of type \ref{d=3:D4}. Then $\Aut^0(X)=\Gm$ and $\Aut(\Gamma(X))$ is isomorphic to $\DD_3$. The surface can be given by the equation $x_0^2x_3=x_1x_2(x_1+x_2)$ in $\PP^3$. $\Aut^0(X)$ consists of the transformations 
    \begin{eqnarray*}
    (x_0:x_1:x_2:x_3)&\longmapsto& (\lambda x_0:x_1 :x_2:x_3/\lambda^2).
    \end{eqnarray*}
    The transformations
    \begin{eqnarray*}
    (x_0:x_1:x_2:x_3)&\longmapsto& (x_0:x_2:x_1:x_3),
    \\
    (x_0:x_1:x_2:x_3)&\longmapsto& (x_0:x_2:-(x_1+x_2):x_3)
    \end{eqnarray*}
    generate a subgroup $G\subset\Aut(X)$. Notice that $G$ is isomorphic to $\DD_3$. Since the action of $G$ on the set of the lines is faithful, we see that the natural homomorphism $G\to\Aut(\Gamma(X))$ is an isomorphism. Thus, $\Aut(X)=\Gm\times\DD_3$.

    Similarly, if $X$ is of type \ref{d=1:E6-A2} or \ref{d=4:A3-5lines} then the natural homomorphism is surjective and has a section. Moreover, $\Aut(X)=\Gm\times\ZZ/2\ZZ$.

    \subsection*{Case \ref{d=3:3A2}} Let $X$ be the surface of type \ref{d=3:3A2}. Let $\{E_1,E_2,E_3\}$ be the set of all $(-1)$-curves on $\widetilde{X}$. Then the set $\{E_1,E_2,E_3\}$ is $\Aut(\Gamma(X))$-invariant. There exists a simultaneous contraction  $\varphi:\widetilde{X}\to\widetilde{Y}$ of $(-1)$-curves $E_1,E_2,E_3$, where $Y$ is the surface \ref{d=6:smooth} and $\varphi(E_i)=P_i=E_i'\cap E_i''$, where $E_i'$ and $E_i''$ are curves with negative self-intersection numbers on the surface $\widetilde{Y}$. Notice that $$\Aut(\Gamma(X))\longrightarrow\Aut(\Gamma(Y))$$ is well-defined and injective. Now, let $G\subset \Aut(Y)$ be the subgroup $\Aut^0(Y)\rtimes\Aut(\Gamma(X))$. Then $\Aut(X)=\Aut(\Bl_{P_1,P_2,P_3}(\widetilde{Y}))=G$. Thus, $\Aut(\Gamma(X))$-action lifts from $Y$ to $X$, so we are done.
    
     Similarly, let $X$ be the surface of type \ref{d=2:D4-3A1}, \ref{d=2:2A3-A1}, \ref{d=2:D5-A1}, \ref{d=3:A3-2A1}, \ref{d=3:2A2-A1}, \ref{d=4:A3-2A1}, \ref{d=4:A3-A1}, \ref{d=4:A2-2A1} or \ref{d=4:A2-A1}. 
    It is easily shown that there exists a $\Aut(\Gamma(X))$-invariant set $\{E_1,\ldots,E_k\}$ of $(-1)$-curves on surface $\widetilde{X}$ such that there exists a simultaneous contraction  $\varphi:\widetilde{X}\to\widetilde{Y}$ of $(-1)$-curves $E_1,\ldots,E_k$, where $Y$ is the surface \ref{d=3:D4}, \ref{d=4:3A1}, \ref{d=3:D4}, \ref{d=5:A1}, \ref{d=4:3A1}, \ref{d=6:A1-4l}, \ref{d=5:A2}, \ref{d=6:smooth} or \ref{d=5:A1}, respectively, and $\varphi(E_i)=P_i=E_i'\cap E_i''$, where $E_i'$ and $E_i''$ are curves with negative self-intersection numbers on the surface $\widetilde{Y}$. It is easily shown that $$\Aut(\Gamma(X))\longrightarrow\Aut(\Gamma(Y))$$ is well-defined and injective. Now, let $G\subset \Aut(Y)$ be the subgroup $\Aut^0(Y)\rtimes\Aut(\Gamma(X))$. Then $\Aut(X)=\Aut(\Bl_{P_1,\ldots,P_k}(\widetilde{Y}))=G$. Thus, $\Aut(\Gamma(X))$-action lifts from $Y$ to $X$, so we are done. Let us remark that if $\Aut^0(X)=\Aut^0(Y)$ then the structure of semidirect product does not change.

    Let $X$ be the surface of type \ref{d=2:A5-A2}, \ref{d=2:E6} or \ref{d=4:D4}. Let $Y$ be the surface of type \ref{d=1:E6-A2}, \ref{d=1:E6-A2} or \ref{d=2:E6}, respectively. Then there exists $\Aut(Y)$-equivariant contraction $\varphi: \widetilde{Y}\to \widetilde{X}$ such that $$\Aut(\Gamma(Y))\longrightarrow\Aut(\Gamma(X))$$ is an isomorphism, so we are done.
\end{proof}

The two following propositions gives us a description of the automorphism groups in the remaining cases.

\begin{proposition}
    \label{c5}
    Let $X$ be the surface of type \ref{d=1:2D4}, \ref{d=2:2A3}, \ref{d=3:2A2} or
    \ref{d=4:2A1-8lines}. Then $\Aut(X)$ is isomorphic to

    Case \ref{d=1:2D4}
\begin{itemize}
\item
$(\Aut^0(X)\rtimes\ZZ/2\ZZ)\rtimes\ZZ/2\ZZ$ for $\lambda \in\{-1, \frac{1}{2}, 2\} $,
\item
$\Aut^0(X)\rtimes\ZZ/6\ZZ$ for $\lambda=\lambda^2+1$, 
\item
$\Aut^0(X)\rtimes\ZZ/2\ZZ$ otherwise.
\end{itemize}

Case \ref{d=2:2A3}
\begin{itemize} 
\item
$(\Aut^0(X)\rtimes (\ZZ/2\ZZ)^2)\rtimes\ZZ/2\ZZ$ for $\lambda = -1 $, 
\item
$\Aut^0(X)\rtimes (\ZZ/2\ZZ)^2$ otherwise,\end{itemize}

Case \ref{d=3:2A2}
\begin{itemize}
\item
$(\Aut^0(X)\rtimes\ZZ/2\ZZ)\rtimes \ZZ/2\ZZ$ for $\lambda \in\{-1, \frac{1}{2}, 2\} $,
\item
$\Aut^0(X)\rtimes\ZZ/6\ZZ$ for $\lambda=\lambda^2+1$, 
\item
$\Aut^0(X)\rtimes\ZZ/2\ZZ$ otherwise.\end{itemize}

Case \ref{d=4:2A1-8lines}
\begin{itemize}
\item
$(\Aut^0(X)\rtimes(\ZZ/2\ZZ)^3)\rtimes\ZZ/2\ZZ$ for $\lambda \in\{-1, \frac{1}{2}, 2\} $,
\item
$(\Aut^0(X)\rtimes(\ZZ/2\ZZ)^3)\rtimes\ZZ/3\ZZ$ for $\lambda=\lambda^2+1$, 
\item
$\Aut^0(X)\rtimes(\ZZ/2\ZZ)^3$ otherwise.\end{itemize}
\end{proposition}

\begin{proof}
    Consider cases \ref{d=1:2D4}, \ref{d=2:2A3}, \ref{d=3:2A2}, 
    \ref{d=4:2A1-8lines} separately.
    \subsection*{Case \ref{d=4:2A1-8lines}}
    Let $X$ be the surface of type \ref{d=4:2A1-8lines}. Then $\Aut^0(X)=\Gm$. The surface can be given by the equation $$y_2y_2'=y_1y_1'(y_1'-y_1)(y_1'-\lambda y_1)$$ for $\lambda\in \KK\backslash \{0,1\}$ in $\PP(1,1,2,2)$. The following transformation permute the singular points
    \begin{eqnarray*}
        (y_1:y_1':y_2:y_2')&\longmapsto& (y_1:y_1':y_2':y_2).
    \end{eqnarray*}
    Since $X$ is quasismooth it follows that any automorphism of $X$ lifts to an automorphism of $\PP(1,1,2,2)$ (see \cite{ess}). Then consider the $\Aut(X)$-equivariant projection to $\PP(1,1)=\PP^1$: $$X\backslash\{P,Q\}\longrightarrow \PP^1,$$ where $P$ and $Q$ are the singular points. This projection maps the set of the four lines in $y_2=0$ onto the set of the points $\{0,1,\lambda,\infty\}$. Hence there is a permutation of the $4$ lines in $y_2=0$  if and only if there is a permutation of the points $\{0,1,\lambda,\infty\}$ in $\PP^1$. Thus, the image of the natural homomorphism $\Aut(X)\to\Aut(\Gamma(X))$ is isomorphic to $\Aut(\PP^1,\{0,1,\lambda,\infty\})\times\ZZ/2\ZZ$. Now, we use the equation of the surface $X$ to get the required assertion. The transformations 
    \begin{eqnarray*}
    \tau_t:(y_1:y_1':y_2:y_2')&\longmapsto& (y_1:y_1':y_2/t:t y_2'),
    \\
    \sigma_1:(y_1:y_1':y_2:y_2')&\longmapsto& (y_1'-\lambda y_1:\lambda y_1'-\lambda y_1:\lambda(\lambda-1)y_2:\lambda(\lambda-1)y_2'),
    \\
    \sigma_2:(y_1:y_1':y_2:y_2')&\longmapsto& (-y_1'+y_1:-y_1'+\lambda y_1:(\lambda-1)y_2:(\lambda-1)y_2'),
    \\
    \sigma_3:(y_1:y_1':y_2:y_2')&\longmapsto& (y_1:y_1':y_2':y_2)
    \end{eqnarray*}
    generate the subgroup $G\subset\Aut(X)$. This subgroup is isomorphic to $\Aut^0(X)\rtimes(\ZZ/2\ZZ)^3$. Notice that this semidirect product is nontrivial since $\sigma_3\cdot t=t^{-1}$.
    
    Consider the subcase, where $\lambda \not\in\{-1, \frac{1}{2}, 2\}$ and $\lambda\neq\lambda^2+1$. Then $\Aut(X)/G$ is trivial. Therefore, $\Aut(X)=G$. Thus, $\Aut(X)$ is isomorphic to $\Aut^0(X)\rtimes(\ZZ/2\ZZ)^3$.

    Now assume that $\lambda\in\{-1, \frac{1}{2}, 2\}.$ Then the corresponding surfaces given by the equation above are isomorphic. Therefore, the automorphism groups of the corresponding surfaces are isomorphic.
    
    Consider the subcase, where $\lambda=-1$. Notice that $\Aut(X)/G=\ZZ/2\ZZ.$ Then there is the transformation 
    \begin{eqnarray*}
        \sigma:(y_1:y_1':y_2:y_2')&\longmapsto& (-y_1:y_1':-y_2:y_2')
    \end{eqnarray*}
    such that $\sigma^2=\operatorname{id}$ and $\sigma\not\in G.$
    Therefore, $\Aut(X)=G\rtimes\ZZ/2\ZZ$. Thus, $\Aut(X)$ is isomorphic to $(\Aut^0(X)\rtimes(\ZZ/2\ZZ)^3)\rtimes\ZZ/2\ZZ$. Moreover, $$\sigma\cdot t=t,\quad \sigma\cdot\sigma_1=\sigma_2, \quad\sigma\cdot\sigma_2=\sigma_1,\quad \sigma\cdot\sigma_3=(-1,1,1,\sigma_3).$$

    Consider the subcase, where $\lambda=\lambda^2+1$. Notice that $\Aut(X)/G=\ZZ/3\ZZ.$ Then there is the transformation 
    \begin{eqnarray*}
        \sigma:(y_1:y_1':y_2:y_2')&\longmapsto& (y_1:(\lambda-1)y_1'+y_1:y_2:y_2')
    \end{eqnarray*}
    such that $\sigma^3=\operatorname{id}$ and $\sigma\not\in G$.
    Therefore, $\Aut(X)=G\rtimes\ZZ/3\ZZ$. Thus, $\Aut(X)$ is isomorphic to $(\Aut^0(X)\rtimes(\ZZ/2\ZZ)^3)\rtimes\ZZ/3\ZZ$. Moreover, $$\sigma\cdot t=t,\quad \sigma\cdot\sigma_1=(1,\sigma_1,\sigma_2,1),\quad \sigma\cdot\sigma_2=\sigma_1,\quad \sigma\cdot\sigma_3=\sigma_3.$$

\subsection*{Case \ref{d=3:2A2}}
Let $X$ be the surface of type \ref{d=3:2A2}. Then $\Aut^0(X)=\Gm$. The surface can be given by the equation $$x_0x_2x_3=x_1(x_1-x_0)(x_1-\lambda x_0)$$ for $\lambda\in \KK\backslash \{0,1\}$ in $\PP^3$. The transformation \begin{eqnarray*}
(x_0:x_1:x_2:x_3)&\longmapsto& (x_0:x_1:x_3:x_2).
\end{eqnarray*} permute the singular points $P=(0:0:0:1)$ and $Q=(0:0:1:0)$. 
Let $$G=\{\varphi\in\Aut(X):\varphi(P)=P\}\subset\Aut(X).$$ Consider $\varphi\in G$. Since $\varphi$ fix the singular points and the equation has no terms $x_i^3$, $x_2^2x_1$, $x_3^2x_1$, $x_1x_2x_3$, where $i\neq 1$, we see that $\varphi$ is represented as a matrix $$\begin{pmatrix} a & 0 & 0 & 0 \\ b & 1 & 0 & 0 \\ 0 & 0 & d & 0 \\ 0 & 0 & 0 & e \end{pmatrix}.$$ Now, we use the equation of the surface $X$ to get the required assertion. The transformations
    \begin{eqnarray*}
    \tau_t:(x_0:x_1:x_2:x_3)&\longmapsto& (x_0:x_1:x_2/t:t x_3),
    \\
    \sigma_1:(x_0:x_1:x_2:x_3)&\longmapsto& (x_0:x_1:x_3:x_2)
    \end{eqnarray*}
    generate the subgroup $H\subset\Aut(X)$. This subgroup is isomorphic to $\Aut^0(X)\rtimes\ZZ/2\ZZ$. Notice that this semidirect product is nontrivial since $\sigma_1\cdot t=t^{-1}$.
    
   Consider the subcase, where $\lambda \not\in\{-1, \frac{1}{2}, 2\}$ and $\lambda\neq\lambda^2+1$. Then $\Aut(X)/G$ is trivial. Therefore, $\Aut(X)=G$. Thus, $\Aut(X)$ is isomorphic to $\Aut^0(X)\rtimes\ZZ/2\ZZ$.

    Now assume that $\lambda\in\{-1, \frac{1}{2}, 2\}.$ Then the corresponding surfaces given by the equation above are isomorphic. Therefore, the automorphism groups of the corresponding surfaces are isomorphic.
    
    Consider the subcase, where $\lambda=-1$. Notice that $\Aut(X)/G=\ZZ/2\ZZ.$ Then there is the transformation 
    \begin{eqnarray*}
        \sigma:(x_0:x_1:x_2:x_3)&\longmapsto& (-x_0:x_1:-x_2:x_3)
    \end{eqnarray*}
    such that $\sigma^2=\operatorname{id}$ and $\sigma\not\in G$.
    Therefore, $\Aut(X)=H\rtimes\ZZ/2\ZZ$. Thus, $\Aut(X)$ is isomorphic to $(\Aut^0(X)\rtimes\ZZ/2\ZZ)\rtimes \ZZ/2\ZZ$. Moreover, $$\sigma\cdot t=t,\quad \sigma\cdot\sigma_1=(-1,\sigma_1).$$
    
    Consider the subcase, where $\lambda=\lambda^2+1$. Notice that $\Aut(X)/G=\ZZ/3\ZZ.$  Then there is the transformation
    \begin{eqnarray*}
    \sigma:(x_0:x_1:x_2:x_3)&\longmapsto& (x_0:(\lambda-1)x_1+x_0:x_3:x_2)
    \end{eqnarray*}
    such that $\sigma^3=\operatorname{id}$ and $\sigma\not\in G$.
    Thus, $\Aut(X)\simeq\Aut^0(X)\rtimes\ZZ/6\ZZ$, so we are done.

In the cases \ref{d=1:2D4} and \ref{d=2:2A3} the embedding of the surface $X$ is canonical (see the proof of Proposition $3.4$ \cite{che}). Hence any automorphism of $X$ lifts to an automorphism of the total space.
\subsection*{Case \ref{d=2:2A3}}
Let $X$ be the surface of type \ref{d=2:2A3}. Then $\Aut^0(X)=\Gm$. The surface can be given by the equation $$y_2^2=(y_1''^2+y_1y_1')(y_1''^2+\lambda y_1y_1')$$ for $\lambda\in \KK\backslash \{0,1\}$ in $\PP(1,1,1,2)$. Consider the $\Aut(X)$-equivariant projection to $\PP(1,1,1)$: $$X\longrightarrow \PP^2.$$ Let $$Q_1=\{y_1''^2+y_1y_1'=0\},\quad Q_2=\{y_1''^2+\lambda y_1y_1'=0\}$$ in $\PP^2$. Then $$\Aut(X)=\Aut(\PP^2,\{Q_1,Q_2\})\times\ZZ/2\ZZ.$$  There is a permutation of the quadrics $Q_1$ and $Q_2$ if and only if $\lambda=-1$. We note that $Q_1\cap Q_2$ is the two-point set $\{P_1,P_2\}$. Then notice that $\Aut(\PP^2,\{Q_1,P_1,P_2\})$ is isomorphic to $\Gm\rtimes\ZZ/2\ZZ$. Now, we use the equation of the surface $X$ to get the required assertion. The transformations 
    \begin{eqnarray*}
    \tau_t:(y_1:y_1':y_1'':y_2)&\longmapsto& (y_1/t:t y_1':y_1'':y_2),
    \\
    \sigma_1:(y_1:y_1':y_1'':y_2)&\longmapsto& (y_1':y_1:y_1'':y_2),
    \\
    \sigma_2:(y_1:y_1':y_1'':y_2)&\longmapsto& (y_1:y_1':y_1'':-y_2)
    \end{eqnarray*}
    generate the subgroup $H\subset\Aut(X)$. This subgroup is isomorphic to $\Aut^0(X)\rtimes(\ZZ/2\ZZ)^2$. Notice that this semidirect product is nontrivial since $\sigma_1\cdot t= t^{-1}$.
    
    Consider the subcase, where $\lambda \neq -1$. Then $\Aut(X)/G$ is trivial. Therefore, $\Aut(X)=G$. Thus, $\Aut(X)$ is isomorphic to $\Aut^0(X)\rtimes(\ZZ/2\ZZ)^2$.

    Consider the subcase, where $\lambda=-1$. Notice that $\Aut(X)/G=\ZZ/2\ZZ.$ Then there is the transformation 
    \begin{eqnarray*}
    \sigma:(y_1:y_1':y_1'':y_2)&\longmapsto& (-y_1:y_1':y_1'':y_2)
    \end{eqnarray*} such that $\sigma^2=\operatorname{id}$ and $\sigma\not\in G$.
    Therefore, $\Aut(X)=H\rtimes\ZZ/2\ZZ$. Thus, $\Aut(X)$ is isomorphic to $(\Aut^0(X)\rtimes (\ZZ/2\ZZ)^2)\rtimes\ZZ/2\ZZ$. Moreover, $$\sigma\cdot t=t,\quad \sigma\cdot\sigma_1=(-1,\sigma_1,1),\quad \sigma\cdot\sigma_2=\sigma_2.$$

\subsection*{Case \ref{d=1:2D4}}
Let $X$ be the surface of type \ref{d=1:2D4}. Then $\Aut^0(X)=\Gm$. The surface can be given
by the equation $$y_3^2=y_2(y_2+y_1y_1')(y_2+\lambda y_1y_1')$$ for $\lambda\in \KK\backslash \{0,1\}$ in $\PP(1,1,2,3)$. The following transformation permute the singular points
\begin{eqnarray*}
(y_1:y_1':y_2:y_3)&\longmapsto& (y_1':y_1:y_2:y_3).
\end{eqnarray*}
Then consider the $\Aut(X)$-equivariant projection to $\PP(1,1,2)$: $$X\longrightarrow \PP(1,1,2).$$ This projection maps the three lines in $y_3=0$ onto the three lines $$y_2=0,\quad y_2+y_1y_1'=0,\quad y_2+\lambda y_1y_1'=0$$ in $\PP(1,1,2)$. Hence there is a permutation of the three lines in $y_3=0$  if and only if there is a permutation of the above three lines in $\PP(1,1,2)$. Now, we use the equation of the surface $X$ to get the required assertion. The transformations 
\begin{eqnarray*}
\tau_t:(y_1:y_1':y_2:y_3)&\longmapsto& (y_1/t:t y_1':y_2:y_3),
\\
\sigma_1:(y_1:y_1':y_2:y_3)&\longmapsto& (y_1':y_1:y_2:y_3)
\end{eqnarray*}
generate the subgroup $H\subset\Aut(X)$. This subgroup is isomorphic to $\Aut^0(X)\rtimes\ZZ/2\ZZ$. Notice that this semidirect product is nontrivial since $\sigma_1\cdot t= t^{-1}$.
    
   Consider the subcase, where $\lambda \not\in\{-1, \frac{1}{2}, 2\}$ and $\lambda\neq\lambda^2+1$. Then $\Aut(X)/G$ is trivial. Therefore, $\Aut(X)=G$. Thus, $\Aut(X)$ is isomorphic to $\Aut^0(X)\rtimes\ZZ/2\ZZ$.

    Now assume that $\lambda\in\{-1, \frac{1}{2}, 2\}.$ Then the corresponding surfaces given by the equation above are isomorphic. Therefore, the automorphism groups of the corresponding surfaces are
isomorphic.
    
    Consider the subcase, where $\lambda=-1$. Notice that $\Aut(X)/G=\ZZ/2\ZZ.$ Then there is the transformation 
    \begin{eqnarray*}
        \sigma:(y_1:y_1':y_2:y_3)&\longmapsto& (-y_1:y_1':y_2:y_3)
    \end{eqnarray*}
    such that $\sigma^2=\operatorname{id}$ and $\sigma\not\in G.$ 
    Therefore, $\Aut(X)=H\rtimes\ZZ/2\ZZ$. Thus, $\Aut(X)$ is isomorphic to $(\Aut^0(X)\rtimes\ZZ/2\ZZ)\rtimes\ZZ/2\ZZ$. Moreover, $$\sigma\cdot t=t,\quad \sigma\cdot\sigma_1=(-1,\sigma_1).$$

    Consider the subcase, where $\lambda=\lambda^2+1$. Notice that $\Aut(X)/G=\ZZ/3\ZZ.$  Then there is the transformation
    \begin{eqnarray*}
    \sigma:(y_1:y_1':y_2:y_3)&\longmapsto& (-\lambda y_1':-\lambda y_1:y_2+y_1y_1':y_3).
    \end{eqnarray*}
    such that $\sigma^3=\operatorname{id}$ and $\sigma\not\in G$.
    Thus, $\Aut(X)\simeq\Aut^0(X)\rtimes\ZZ/6\ZZ$, so we are done.
\end{proof}

\begin{proposition}
    \label{c6}
Let $X$ be a surface of type \ref{d=2:A7}, \ref{d=3:A5} or \ref{d=4:A3-4lines}.
    \begin{itemize}
     \item 
If $X$ is of type~\ref{d=2:A7}, then $\Aut(X)\simeq(\Aut^0(X)\rtimes\ZZ/4\ZZ)\times\ZZ/2\ZZ$.
     \item 
If $X$ is of type~\ref{d=3:A5}, then $\Aut(X)\simeq\Aut^0(X)\rtimes \ZZ/6\ZZ$.
     \item 
If $X$ is of type~\ref{d=4:A3-4lines}, then $\Aut(X)\simeq\Aut^0(X)\rtimes((\ZZ/2\ZZ)^2\rtimes\ZZ/4\ZZ)$.
 \end{itemize}
\end{proposition}

\begin{proof}
Consider cases \ref{d=2:A7}, \ref{d=3:A5}, \ref{d=4:A3-4lines} separately.
\subsection*{Case \ref{d=2:A7}}
Let $X$ be the surface of type \ref{d=2:A7}. It is readily seen that $\Aut(X)=H\rtimes\ZZ/2\ZZ$, where $H$ is the kernel of the natural homomorphism $\Aut(X)\to\Aut(\Gamma(X))$. 
  Let $\widetilde{X}\to \widetilde{Y}$ be the blow-down of the two $(-1)$-curves.
    We see that $Y$ is the surface of type \ref{d=4:A3-2A1}. Further note that 
\[
    H=\{\varphi\in\Aut(Y): \varphi|_{E_1}=\operatorname{id}\quad \text{and}\quad \varphi|_{E_2}=\operatorname{id}\},
\]
where $E_1$ and $E_2$ are the only lines on $Y$. Recall that $\Aut^0(Y)$ is isomorphic to $\GG_a\rtimes\Gm^2$ and $\GG_a$-action on $E_1$ and $E_2$ is trivial. Let $T\subset \Aut^0(Y)$ be a two-dimensional torus. Hence $H=\GG_a\rtimes (T\cap H)$. Since we know the structure of the toric fan of $\widetilde{Y}$, we see that $T\cap H$ is isomorphic to $\ZZ/4\ZZ$. Finally, we use the explicit equation of the surface \ref{d=2:A7} to get the required assertion. The transformations 
\begin{eqnarray*}
\alpha_a:(y_1:y_1':y_1'':y_2)&\longmapsto& (y_1+ay_1':y_1':y_1''-2ay_1-a^2y_1':y_2), 
\\
\sigma_1:(y_1:y_1':y_1'':y_2)&\longmapsto& (y_1:y_1':y_1'':-y_2),
    \\
\sigma_2:(y_1:y_1':y_1'':y_2)&\longmapsto &(y_1:\varepsilon y_1':-\varepsilon y_1'':-y_2),
\end{eqnarray*}
where $\varepsilon^2=-1$, generate the group $\Aut(X).$ Thus, $$\Aut(X)\simeq(\Aut^0(X)\rtimes\ZZ/4\ZZ)\times\ZZ/2\ZZ.$$ Moreover, $\sigma_2\cdot a=-\varepsilon a$.
\subsection*{Case \ref{d=3:A5}}
 Let $X$ be the surface of type \ref{d=3:A5}. It is readily seen that $\Aut(X)=H\rtimes\ZZ/2\ZZ$, where $H$ is the kernel of the natural homomorphism $\Aut(X)\to\Aut(\Gamma(X))$. 
  Let $\widetilde{X}\to \widetilde{Y}$ be the blow-down of the two $(-1)$-curves, where $Y$ is of type \ref{d=5:A2-A1}.
    The dual graph of $Y$ is as follows 
    
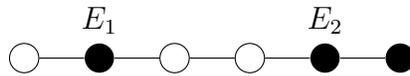
\begin{center}
    \begin{tikzpicture}
	\node[circle,fill=black] (1) at (1,0) {};
	\node[circle,fill=black] (2) at (0,0) {};
        \node[circle,draw=black] (3) at (-1,0) {};
        \node[circle,draw=black] (4) at (-2,0) {};
        \node[circle,fill=black] (5) at (-3,0) {};
        \node[circle,draw=black] (6) at (-4,0) {};
        \node[] () at (0,0.5) {$E_2$};
        \node[] () at (-3,0.5) {$E_1$};
	\path (1) edge (2);
        \path (2) edge (3);
        \path (3) edge (4);
        \path (4) edge (5);
        \path (5) edge (6);
    \end{tikzpicture}
    \end{center}
    Further note that 
\[
    H=\{\varphi\in\Aut(Y): \varphi|_{E_1}=\operatorname{id}\quad \text{and}\quad \varphi|_{E_2}=\operatorname{id}\}.
\]
 Recall that $\Aut^0(Y)$ is isomorphic to $\BB_2\times\Gm$ and $\GG_a$-action on $E_1$ and $E_2$ is trivial. Let $T\subset \Aut^0(Y)$ be a two-dimensional torus. Hence $H=\GG_a\rtimes (T\cap H)$. Since we know the structure of the toric fan of $\widetilde{Y}$, we see that $T\cap H$ is isomorphic to $\ZZ/3\ZZ$. Finally, we use the explicit equation of the surface \ref{d=3:A5} to get the required assertion. The transformations 
\begin{eqnarray*}
\alpha_a:(x_0:x_1:x_2:x_3)&\longmapsto& (x_0:x_1:x_2+ax_1:x_3+2ax_2+a^2x_1), 
\\
\sigma:(x_0:x_1:x_2:x_3)&\longmapsto &(x_0:\varepsilon x_1:-x_2:\varepsilon^2x_3),
\end{eqnarray*}
where $\varepsilon^3=1$ and $\varepsilon\neq 1$, generate the group $\Aut(X).$ Thus, $$\Aut(X)\simeq\Aut^0(X)\rtimes\ZZ/6\ZZ.$$ Moreover, $\sigma\cdot a=-\varepsilon^2a$.

  \subsection*{Case \ref{d=4:A3-4lines}} 
Let $X$ be the surface of type \ref{d=4:A3-4lines}. 
 Let  $\widetilde{X}\to \widetilde{Y}$ be the blow-down of the four $(-1)$-curves.
We see that $Y$ is the surface from the case \ref{d=8}. Then we use the structure of the automorphism group of $\PP(1,1,2)$ to get the required assertion. 
Let $G\subset\PGL_3(\KK)$ be  the subgroup 
generated by the transformations 
\begin{eqnarray*}
\alpha_a:(x_0:x_1:x_2:x_3)&\longmapsto& (x_0:x_1:x_2:x_3+ax_2),
\\
\sigma_1:(x_0:x_1:x_2:x_3)&\longmapsto& (x_1:x_0:x_2:x_3),
\\
\sigma_2:(x_0:x_1:x_2:x_3)&\longmapsto& (-x_0:-x_1:x_2:x_3),
\\
\sigma_3:(x_0:x_1:x_2:x_3)&\longmapsto& (x_0:-x_1:\varepsilon x_2:x_3),
\end{eqnarray*}
where $\varepsilon^2=-1$.
Notice that $G\subset\Aut(Y)$, where $Y=\{x_0x_1=x_2^2\}$ in $\PP^3$. This subgroup is isomorphic to $\Aut(X)$. Therefore, $\Aut(X)$ is isomorphic to $\Aut^0(X)\rtimes((\ZZ/2\ZZ)^2\rtimes\ZZ/4\ZZ)$. Moreover, $$\sigma_1\cdot a=a,\quad \sigma_2\cdot a=a, \sigma_3\cdot a=-\varepsilon a,\quad \sigma_3\cdot \sigma_1=(\sigma_1,\sigma_2),\quad \sigma_3\cdot \sigma_2=\sigma_2.$$
This completes the proof of
the proposition.
\end{proof}

\section{Actions of groups of connected components}

In this section we describe the actions of $\Aut(X)/\Aut^0(X)$ on $\Aut^0(X)$ for the surfaces of types \ref{d=1:E8}-\ref{d=6:A1-4l}, \ref{d=7} except for the cases \ref{d=1:2D4}, \ref{d=2:A7}, \ref{d=2:2A3}, \ref{d=3:A5}, \ref{d=3:2A2}, \ref{d=4:A3-4lines} and \ref{d=4:2A1-8lines}. Descriptions of actions in these seven cases are given in proofs of the Proposition \ref{c5} and Proposition \ref{c6}.

\begin{claim}
    Let $X$ be a surface of type \ref{d=1:E6-A2}, \ref{d=2:A5-A2}, \ref{d=2:D4-3A1}, \ref{d=2:E6}, \ref{d=2:D5-A1}, \ref{d=3:D4}, \ref{d=4:A3-5lines}, \ref{d=4:A2-A1} or \ref{d=5:A1}. Then $\Aut(X)$ is isomorphic to $\Aut^0(X)\times\Aut(\Gamma(X))$.
\end{claim}

\begin{proof}
    We use the explicit equations of the surfaces.
    
    For example, if $X$ is of type \ref{d=2:A5-A2}, then the surface can be given
by the equation $$y_2^2=y_1''(y_1^2y_1''+y_1'^3)$$ in $\PP(1,1,1,2)$. Then $\Aut^0(X)=\Gm$ and $\Aut(\Gamma(X))=\ZZ/2\ZZ$. The transformations
    \begin{eqnarray*}
        (y_1:y_1':y_1'':y_2)&\longmapsto& (y_1/\lambda^3:y_1'/\lambda:\lambda^3 y_1'':y_2),
        \\
        (y_1:y_1':y_1'':y_2)&\longmapsto& (y_1:y_1':y_1'':-y_2)    
    \end{eqnarray*}
     generate the group $\Aut(X)$. Thus, $\Aut(X)$ is isomorphic to $\Gm\times\ZZ/2\ZZ$.
     
     Except for the cases \ref{d=4:A2-A1}, \ref{d=5:A1} the rest of the proof is done similarly. For the cases \ref{d=4:A2-A1}, \ref{d=5:A1} we refer to the proofs of Lemma \ref{lm} and Proposition \ref{cm}.
\end{proof}

\begin{claim}
Let $X$ be a surface of type \ref{d=2:2A3-A1}, \ref{d=3:A3-2A1}, \ref{d=3:2A2-A1} or \ref{d=4:3A1}. Then $\Aut^0(X)=\Gm$ and the automorphism group $\Aut(X)$ is a nontrivial semidirect product of $\Aut^0(X)$ and $\Aut(\Gamma(X))$.
\end{claim}
\begin{proof}
    We use the explicit equations of the surfaces.
    
    For example, if $X$ is the surface \ref{d=2:2A3-A1}, then the surface can be given
by the equation $$y_2^2=y_1y_1'(y_1y_1'+y_1''^2)$$ in $\PP(1,1,1,2)$. Then $\Aut^0(X)=\Gm$ and $\Aut(\Gamma(X))=(\ZZ/2\ZZ)^2$. The transformations
    \begin{eqnarray*}
    (y_1:y_1':y_1'':y_2)&\longmapsto& (\lambda y_1:y_1'/\lambda:y_1'':y_2), 
    \\
    (y_1:y_1':y_1'':y_2)&\longmapsto& (y_1':y_1:y_1'':y_2),
    \\
    (y_1:y_1':y_1'':y_2)&\longmapsto& (y_1:y_1':y_1'':-y_2)
    \end{eqnarray*}
     generate the group $\Aut(X)$. Thus, the automorphism group $\Aut(X)\simeq\Gm\rtimes(\ZZ/2\ZZ)^2$ and this semidirect product is nontrivial. The rest of the proof is done similarly.
\end{proof}

\begin{claim}
    Let $X$ be a surface of type  \ref{d=3:3A2}. Consider the group $\DD_3$ generated by $\sigma_2$ and $\sigma_3$, where $\sigma_2^2=\operatorname{id}$ and $\sigma_3^3=\operatorname{id}$. Define the action of $\DD_3$ on $\Gm^2$ as follows
    \begin{eqnarray*}   
        &&\sigma_2\cdot_\varphi (t_1,t_2)=((t_1t_2)^{-1},t_2),\\ &&\sigma_3\cdot_\varphi (t_1,t_2)=(t_2,(t_1t_2)^{-1}),\quad \text{where} \quad (t_1,t_2)\in\Gm^2.   
    \end{eqnarray*}
     Then $\Aut(X)$ is isomorphic to $\Gm^2\rtimes_\varphi\DD_3$.
\end{claim}
\begin{proof}
    The surface can be given by the equation $x_0x_1x_2=x_3^3$ in $\PP^3$. The transformations
    \begin{eqnarray*}
    \tau_{t_1}:( x_0:x_1:x_2:x_3)&\longmapsto& (t_1 x_0:x_1/t_1:x_2:x_3)
    \\
    \tau'_{t_2}:(x_0: x_1:x_2:x_3)&\longmapsto& (t_2x_0: x_1:x_2/t_2:x_3)
    \\
    \sigma_2:(x_0:x_1:x_2:x_3)&\longmapsto& (x_1:x_0:x_2:x_3)
    \\
    \sigma_3:(x_0:x_1:x_2:x_3)&\longmapsto& (x_1:x_2:x_0:x_3)
    \end{eqnarray*}
    generate the group $\Aut(X)$, so we are done.
\end{proof}

The following two claims can be proved similarly.

\begin{claim}
    Let $X$ be a surface of type \ref{d=4:A2-2A1}. Consider the group $\ZZ/2\ZZ$ generated by $\sigma_2$. Define the action of $\ZZ/2\ZZ$ on $\Gm^2$ as follows $$\sigma_2\cdot_\varphi (t_1,t_2)=(t_1^{-1},t_2),\quad\text{where}\quad  (t_1,t_2)\in\Gm^2.$$ Then $\Aut(X)$ is isomorphic to $\Gm^2\rtimes_\varphi\ZZ/2\ZZ$.
\end{claim}

\begin{claim}
    Let $X$ be a surface of type \ref{d=5:2A1}. Consider the group $\ZZ/2\ZZ$ generated by $\sigma_2$. Define the action of $\ZZ/2\ZZ$ on $\Gm^2$ as follows $$\sigma_2\cdot_\varphi (t_1,t_2)=(t_1,(t_1t_2)^{-1}),\quad\text{where}\quad (t_1,t_2)\in\Gm^2.$$ Then $\Aut(X)$ is isomorphic to $\Gm^2\rtimes_\varphi\ZZ/2\ZZ$.
\end{claim}

\begin{remark}
There are surfaces $X$ such that there exists automorphisms $\varphi\in\Aut(X)$ that cannot be lifted to the automorphisms of the total space.
For example, let $X$ be the surface of type \ref{d=6:A2}. Then the surface can be given by the equation $y_1y_3=y_2^2+y_1'^4$ in $\PP(1,1,2,3)$. Now we prove the required assertion. Let $E_1$ and $E_2$ be the lines on $X$. Notice that $E_1=\{y_1=y_1'=0\}$ and $E_2=\{y_1=y_2=0\}$. Consider the automorphism $\psi$ of $\PP(1,1,2,3)$. It is clear that $\psi(E_1)=E_1$. However there is $\varphi\in\Aut(X)$ such that $\varphi(E_1)=E_2$, so we are done.
\end{remark}

\begin{claim}
    Let $X$ be a surface of type \ref{d=4:4A1}. Consider the group $\DD_4$ generated by $\sigma_2$ and $\sigma_4$, where $\sigma_2^2=\operatorname{id}$ and $\sigma_4^4=\operatorname{id}$. Define the action of $\DD_4$ on $\Gm^2$ as follows
    \begin{eqnarray*}
    &&\sigma_2\cdot_\varphi (t_1,t_2)=(t_1,t_2^{-1}),\\ &&\sigma_4\cdot_\varphi (t_1,t_2)=(t_2,t_1^{-1}),\quad \text{where}\quad  (t_1,t_2)\in\Gm^2. 
    \end{eqnarray*}
    Then $\Aut(X)$ is isomorphic to $\Gm^2\rtimes_\varphi\DD_4$.
\end{claim}
\begin{proof}
    We notice that $\widetilde{X}$ is isomorphic to $\Bl_{P, P', Q, Q'}(\PP^1\times\PP^1)$, where $P=(0,0)$, $P'=(0,1)$, $Q=(1,0)$, and $Q'=(1,1)$ are points in $\PP^1\times\PP^1$. Thus, $$\Aut(X)\simeq\Aut(\PP^1\times\PP^1,\{P,P',Q,Q'\}).$$ Then the claim easily follows.
\end{proof}

The following two claims can be proved similarly

\begin{claim}
    Let $X$ be a surface of type \ref{d=6:A1-3l}. Consider the group $\DD_3$ generated by $\sigma_2$ and $\sigma_3$, where $\sigma_2^2=\operatorname{id}$, $\sigma_3^3=\operatorname{id}$. Define the action of $\DD_3$ on $\Ga^2\rtimes\Gm$ as follows
    \begin{eqnarray*}
    &&\sigma_2\cdot_\varphi (a_1,a_2,t)=(a_2,a_1,t),\\ &&\sigma_3\cdot_\varphi (a_1,a_2,t)=(a_2,-a_1-a_2,t),\quad \text{where}\quad (a_1,a_2,t)\in\GG_a^2\rtimes\Gm.
    \end{eqnarray*}
    Then $\Aut(X)$ is isomorphic to $(\GG_a^2\rtimes\Gm)\rtimes_\varphi\DD_3$
\end{claim}

\begin{claim}
    Let $X$ be a surface of type \ref{d=6:A2}. Consider the group $\ZZ/2\ZZ$ generated by $\sigma_2$. Define the action of $\ZZ/2\ZZ$ on $\UUU_3\rtimes\Gm$ as follows $$\sigma_2\cdot_\varphi (a_1,a_2,a_3,t)=(-a_1,-a_2,a_3,t),\quad\text{where}\quad (a_1,a_2,a_3,t)\in\UU_3\rtimes\Gm.$$ Then $\Aut(X)$ is isomorphic to $(\UU_3\rtimes\Gm)\rtimes_\varphi\ZZ/2\ZZ$.
\end{claim}

\begin{claim}
    Let $X$ be a surface of type \ref{d=4:A3-2A1} or \ref{d=6:A1-4l}. Consider the group $\ZZ/2\ZZ$ generated by $\sigma_2$. Define the action of $\ZZ/2\ZZ$ on $\BB_2\times\Gm$ as follows $$\sigma_2\cdot_\varphi (a,t_1,t_2)=(a,t_1,t_2^{-1}),\quad\text{where}\quad (a,t_1,t_2)\in\BB_2\times\Gm.$$ Then $\Aut(X)$ is isomorphic to $(\BB_2\times\Gm)\rtimes_\varphi\ZZ/2\ZZ$.
\end{claim}
\begin{proof}
    Notice that the automorphism group from the case \ref{d=4:A3-2A1} is isomorphic to the automorphism group from the case \ref{d=6:A1-4l} (see the proof of the Proposition \ref{cm}). Let $X$ be the surface of type \ref{d=6:A1-4l}. We notice that $\widetilde{X}$ is isomorphic to $\Bl_{P, P'}(\PP^1\times\PP^1)$, where $P=(0,0)$ and $P'=(0,1)$ are points in $\PP^1\times\PP^1$. Then the claim easily follows.
\end{proof}

\begin{claim}
    Let $X$ be a surface of type \ref{d=4:A3-A1} or \ref{d=5:A2}. Consider the group $\ZZ/2\ZZ$ generated by $\sigma_2$. Define the action of $\ZZ/2\ZZ$ on $\BB_2$ as follows $$\sigma_2\cdot_\varphi (a,t)=(-a,t),\quad\text{where}\quad (a,t)\in\BB_2.$$ Then $\Aut(X)$ is isomorphic to $\BB_2\rtimes_\varphi\ZZ/2\ZZ$.
\end{claim}
\begin{proof}
    Notice that the automorphism group from the case \ref{d=4:A3-A1} is isomorphic to the automorphism group from the case \ref{d=5:A2}. Let $X$ be a surface of type \ref{d=4:A3-A1}. Then the surface can be given by the equation $y_3^2=y_1^6+y_2y_4$ in $\PP(1,2,3,4)$. The transformations 
    \begin{eqnarray*}
    \alpha_a:( y_1:y_2:y_3:y_4)&\longmapsto& ( y_1:y_2:y_3+ay_1y_2:y_4+2ay_1y_3+a^2y_1^2y_2),
    \\
    \tau_t:(y_1: y_2:y_3:y_4)&\longmapsto& (y_1:ty_2:y_3:y_4/t),
    \\
    \sigma_2:(y_1:y_2:y_3:y_4)&\longmapsto& (y_1:y_2:-y_3:y_4)
    \end{eqnarray*}
    generate the group $\Aut(X)$. Then the claim easily follows.
\end{proof}

\begin{claim}
    Let $X$ be a surface of type \ref{d=4:D4}. Consider the group $\ZZ/2\ZZ$ generated by $\sigma_2$. Define the action of $\ZZ/2\ZZ$ on $\Ga\rtimes_{(2)}\Gm$ as follows $$\sigma_2\cdot_\varphi (a,t)=(-a,t),\quad\text{where}\quad (a,t)\in\Ga\rtimes_{(2)}\Gm.$$ Then $\Aut(X)$ is isomorphic to $(\Ga\rtimes_{(2)}\Gm)\rtimes_\varphi\ZZ/2\ZZ$.
\end{claim}
\begin{proof}
    Let $X$ be the surface of type \ref{d=4:D4}.
    Then the surface can be given by the equation $y_2^2=y_2'y_1^2+y_1'^4$ in $\PP(1,1,2,2)$. The transformations
    \begin{eqnarray*}
    \alpha_a:( y_1:y_1':y_2:y_2')&\longmapsto& ( y_1:y_1':y_2+ay_1^2:y_2'+2ay_2+a^2y_1^2),
    \\
    \tau_t:( y_1:y_1':y_2:y_2')&\longmapsto& ( ty_1:y_1':y_2:y_2'/t^2),
    \\
    \sigma_2:( y_1:y_1':y_2:y_2')&\longmapsto& ( y_1:y_1':-y_2:y_2')
    \end{eqnarray*}
    generate the group $\Aut(X)$. Then the claim easily follows.
\end{proof}

\newpage
\appendix
\begin{landscape}
\section{Table}
\label{section:tables}
\pagestyle{empty}
In the table below $X$ is a Du Val del Pezzo surface such that $\Aut(X)$ is infinite.
Then the type $\Type(X)$, the degree $K_X^2$, the Picard rank $\uprho(X)$, the number of lines $\numl(X)$,
the Fano--Weil index $\ind(X)$, the group $\Aut^0(X)$, the group $\Aut(X)$, and the equation of the surface $X$ are given below.
The column No indicates a del Pezzo surface from which $X$ can be obtained by blowing up a smooth point that
does not lie on a line. Except for the column $\Aut(X)$, all information in the table is taken from \cite{che}. For short, we will denote $\Aut^0(X)$ by $G^0$.

\begin{center}
\begin{longtable}{|c|c|c|c|c|c|c|c|c|p{0.3\textwidth}|c|}\hline
&$K_X^2$&$\uprho(X)$&$\numl(X)$&$\Type(X)$&$\ind(X)$&No&$\Aut^0(X)$&$\Aut(X)$&\multicolumn2{c|}{equation \& total space}
\\
\hhline{|=|=|=|=|=|=|=|=|=|=|=|}
\endhead\hline
&$K_X^2$&$\uprho(X)$&$\numl(X)$&$\Type(X)$&$\ind(X)$&No&$\Aut^0(X)$&$\Aut(X)$&\multicolumn2{c|}{equation \& total space}
\\
\hhline{|=|=|=|=|=|=|=|=|=|=|=|}
\endfirsthead
\hline
\no\label{d=1:E8}
&$1$&$1$&$1$&$\type{E}_8$&$1$&--&$\Gm$&$G^0$&\centering$y_3^2=y_2^3+y_1'y_1^5$&$\PP(1^2,2,3)$
\\
\hline
\no\label{d=1:E7-A1}
&$1$&$1$&$3$&$\type{E}_{7}\type{A}_1$&$1$&--&$\Gm$&$G^0$&\centering$y_3^2=y_1^3y_1'y_2+y_2^3$&$\PP(1^2,2,3)$
\\
\hline
\no\label{d=1:E6-A2}
&$1$&$1$&$4$&$\type{E}_6\type{A}_2$&$1$&--&$\Gm$&$G^0\times\ZZ/2\ZZ$&\centering$y_3^2=y_2^3+y_1'^2y_1^4$&$\PP(1^2,2,3)$
\\
\hline
\no\label{d=1:2D4}
&$1$&$1$&$5$&$2\type{D}_4$&$1$&--&$\Gm$&see Pr. \ref{c5}&\centering$y_3^2=y_2(y_2+y_1y_1')(y_2+\lambda y_1y_1')$ \\for $\lambda\in\KK\setminus\{0,1\}$
&$\PP(1^2,2,3)$
\\
\hline
\no\label{d=2:E7}
&$2$&$1$&$1$&$\type{E}_7$&$2$&--&$\Gm$&$G^0$&\centering$y_2^2=y_1(y_1^2y_1''+y_1'^3)$&$\PP(1^3,2)$
\\
\hline
\no\label{d=2:D6-A1}
&$2$&$1$&$2$&$\type{D}_{6}\type{A}_1$&$2$&--&$\Gm$&$G^0$&\centering$y_2^2=y_1y_1'(y_1y_1''+y_1'^2)$&$\PP(1^3,2)$
\\
\hline
\no\label{d=2:A7}
&$2$&$1$&$2$&$\type{A}_7$&$1$&--&$\Ga$&see Pr. \ref{c6}&\centering$y_2^2=(y_1'y_1''+y_1^2)^2-y_1'^4$&$\PP(1^3,2)$
\\
\hline
\no\label{d=2:A5-A2}
&$2$&$1$&$3$&$\type{A}_5\type{A}_2$&$2$&--&$\Gm$&$G^0\times\ZZ/2\ZZ$&\centering$y_2^2=y_1''(y_1^2y_1''+y_1'^3)$&$\PP(1^3,2)$
\\
\hline
\no\label{d=2:D4-3A1}
&$2$&$1$&$4$&$\type{D}_43\type{A}_1$&$2$&--&$\Gm$&$G^0\times\DD_3$&\centering$y_2^2=y_1y_1'y_1''(y_1'+y_1'')$&$\PP(1^3,2)$
\\
\hline
\no\label{d=2:2A3-A1}
&$2$&$1$&$4$&$2\type{A}_3\type{A}_1$&$2$&--&$\Gm$&$G^0\rtimes(\ZZ/2\ZZ)^2$&\centering$y_2^2=y_1y_1'(y_1y_1'+y_1''^2)$&$\PP(1^3,2)$
\\
\hline
\no\label{d=2:E6}
&$2$&$2$&$4$&$\type{E}_6$&$1$&\ref{d=3:E6}&$\Gm$&$G^0\times\ZZ/2\ZZ$&\centering$y_2^2=y_1^3y_1''+y_1'^4$&$\PP(1^3,2)$
\\
\hline
\no\label{d=2:D5-A1}
&$2$&$2$&$5$&$\type{D}_5\type{A}_1$&$2$&--&$\Gm$&$G^0\times\ZZ/2\ZZ$&\centering$y_2^2=y_1'(y_1^2y_1''+y_1'^3)$&$\PP(1^3,2)$
\\
\hline
\no\label{d=2:2A3}
&$2$&$2$&$6$&$2\type{A}_3$&$1$&--&$\Gm$&see Pr. \ref{c5}&\centering$y_2^2=(y_1''^2+y_1y_1')(y_1''^2+\lambda y_1y_1')$ \\for $\lambda\in\KK\setminus\{0,1\}$
&$\PP(1^3,2)$
\\
\hline
\no\label{d=3:E6}
&$3$&$1$&$1$&$\type{E}_6$&$3$&--&$\Ga\rtimes_{(3)}\Gm$&$G^0$&\centering$x_0x_2^2=x_1^3+x_3x_0^2$&$\PP^3$
\\
\hline
\no\label{d=3:A5-A1}
&$3$&$1$&$2$&$\type{A}_5\type{A}_1$&$3$&--&$\BB_2$&$G^0$&\centering$x_0x_2^2=x_1^3+x_0x_3x_1$&$\PP^3$
\\
\hline
\no\label{d=3:3A2}
&$3$&$1$&$3$&$3\type{A}_2$&$3$&--&$\Gm^2$&$G^0\rtimes\DD_3$&\centering$x_0x_1x_2=x_3^3$&$\PP^3$
\\
\hline
\no\label{d=3:D5}
&$3$&$2$&$3$&$\type{D}_5$&$1$&\ref{d=4:D5}&$\Gm$&$G^0$&\centering$x_0^2x_3=x_2(x_0x_2-x_1^2)$&$\PP^3$
\\
\hline
\no\label{d=3:A5}
&$3$&$2$&$3$&$\type{A}_5$&$3$&--&$\Ga$&see Pr. \ref{c6}&\centering$x_0x_2^2=x_1^3+x_0^3+x_0x_3x_1$&$\PP^3$
\\
\hline
\no\label{d=3:A4-A1}
&$3$&$2$&$4$&$\type{A}_4\type{A}_1$&$1$&--&$\Gm$&$G^0$&\centering$x_3(x_0x_2-x_1^2)=x_0^2x_1$&$\PP^3$
\\
\hline
\no\label{d=3:A3-2A1}
&$3$&$2$&$5$&$\type{A}_32\type{A}_1$&$1$&\ref{d=4:A3-2A1}&$\Gm$&$G^0\rtimes\ZZ/2\ZZ$&\centering$x_3(x_0x_2-x_1^2)=x_0x_1^2$&$\PP^3$
\\
\hline
\no\label{d=3:2A2-A1}
&$3$&$2$&$5$&$2\type{A}_2\type{A}_1$&$3$&--&$\Gm$&$G^0\rtimes\ZZ/2\ZZ$&\centering$x_0x_2x_3=x_1^3+x_0x_1^2$&$\PP^3$
\\
\hline
\no\label{d=3:D4}
&$3$&$3$&$6$&$\type{D}_4$&$1$&\ref{d=4:D4}&$\Gm$&$G^0\times\DD_3$&\centering$x_0^2x_3=x_1x_2(x_1+x_2)$&$\PP^3$
\\
\hline
\no\label{d=3:2A2}
&$3$&$3$&$7$&$2\type{A}_2$&$3$&--&$\Gm$&see Pr. \ref{c5}&\centering$x_0x_2x_3=x_1(x_1-x_0)(x_1-\lambda x_0)$ \\for $\lambda\in\KK\setminus\{0,1\}$&$\PP^3$
\\
\hline
\no\label{d=4:D5}
&$4$&$1$&$1$&$\type{D}_5$&$4$&--&$\Ga^2\rtimes\Gm$&$G^0$&\centering$y_3^2=y_2^3+y_1^2y_4$&$\PP(1,2,3,4)$
\\
\hline
\no\label{d=4:A3-2A1}
&$4$&$1$&$2$&$\type{A}_32\type{A}_1$&$4$&--&$\BB_2\times\Gm$&$G^0\rtimes\ZZ/2\ZZ$&\centering$y_3^2=y_2y_4$&$\PP(1,2,3,4)$
\\
\hline
\no\label{d=4:D4}
&$4$&$2$&$2$&$\type{D}_4$&$2$&--&$\Ga\rtimes_{(2)}\Gm$&$G^0\rtimes\ZZ/2\ZZ$&\centering$y_2^2=y_2'y_1^2+y_1'^4$&$\PP(1^2,2^2)$
\\
\hline
\no\label{d=4:A4}
&$4$&$2$&$3$&$\type{A}_4$&$1$&\ref{d=5:A4}&$\BB_2$&$G^0$&\makecell{$x_0x_1-x_2x_3=0$,\\ $x_0x_4+x_1x_2+x_3^2=0$}&$\PP^4$
\\
\hline
\no\label{d=4:A3-A1}
&$4$&$2$&$3$&$\type{A}_3\type{A}_1$&$4$&--&$\BB_2$&$G^0\rtimes\ZZ/2\ZZ$&\centering$y_3^2=y_1^6+y_2y_4$&$\PP(1,2,3,4)$
\\
\hline
\no\label{d=4:A2-2A1}
&$4$&$2$&$4$&$\type{A}_22\type{A}_1$&$2$&--&$\Gm^2$&$G^0\rtimes\ZZ/2\ZZ$&\centering$y_2y_2'=y_1^3y_1'$&$\PP(1^2,2^2)$
\\
\hline
\no\label{d=4:4A1}
&$4$&$2$&$4$&$4\type{A}_1$&$2$&--&$\Gm^2$&$G^0\rtimes\DD_4$&\centering$y_2y_2'=y_1^2y_1'^2$&$\PP(1^2,2^2)$
\\
\hline
\no\label{d=4:A3-4lines}
&$4$&$3$&$4$&$\type{A}_3$&$2$&--&$\Ga$&see Pr. \ref{c6}&\centering$y_2^2=y_2'y_1y_1'+y_1^4+y_1'^4$&$\PP(1^2,2^2)$
\\
\hline
\no\label{d=4:A3-5lines}
&$4$&$3$&$5$&$\type{A}_3$&$1$&\ref{d=5:A3}&$\Gm$&$G^0\times\ZZ/2\ZZ$&\makecell{$x_0x_1-x_2x_3=0$,\\ $ x_0x_3+x_2x_4+x_1x_3=0$}&$\PP^4$
\\
\hline
\no\label{d=4:A2-A1}
&$4$&$3$&$6$&$\type{A}_2\type{A}_1$&$1$&\ref{d=5:A2-A1}&$\Gm$&$G^0\times\ZZ/2\ZZ$&\makecell{$x_0x_1-x_2x_3=0$,\\ $x_1x_2+x_2x_4+x_3x_4=0$}&$\PP^4$
\\
\hline
\no\label{d=4:3A1}
&$4$&$3$&$6$&$3\type{A}_1$&$2$&--&$\Gm$&$G^0\rtimes(\ZZ/2\ZZ)^2$&\centering$y_2y_2'=y_1^2y_1'(y_1'+y_1)$&$\PP(1^2,2^2)$
\\
\hline
\no\label{d=4:2A1-8lines}
&$4$&$4$&$8$&$2\type{A}_1$&$2$&--&$\Gm$&see Pr. \ref{c5}&\centering$y_2y_2'=y_1y_1'(y_1'-y_1)(y_1'-\lambda y_1)$
\\for $\lambda\in\KK\setminus\{0,1\}$&$\PP(1^2,2^2)$
\\
\hline
\no\label{d=5:A4}
&$5$&$1$&$1$&$\type{A}_4$&$5$&--&$\UU_3\rtimes\Gm$&$G^0$&\centering$y_3^2+y_2^3+y_1y_5=0$&$\PP(1,2,3,5)$
\\
\hline
\no\label{d=5:A3}
&$5$&$2$&$2$&$\type{A}_3$&$1$&--&$\Ga^2\rtimes\Gm$&$G^0$&\centering$u_2^2v_0+(u_0^2+u_1u_2)v_1=0$&$\PP^2\times\PP^1$
\\
\hline
\no\label{d=5:A2-A1}
&$5$&$2$&$3$&$\type{A}_2\type{A}_1$&$1$&\ref{d=6:A2-A1}&$\BB_2\times\Gm$&$G^0$&&
\\
\hline
\no\label{d=5:A2}
&$5$&$3$&$4$&$\type{A}_2$&$1$&\ref{d=6:A2}&$\BB_2$&$G^0\rtimes\ZZ/2\ZZ$&\centering$u_0u_1v_0+(u_1^2+u_0u_2)v_1=0$&$\PP^2\times\PP^1$
\\
\hline
\no\label{d=5:2A1}
&$5$&$3$&$5$&$2\type{A}_1$&$1$&\ref{d=6:2A1}&$\Gm^2$&$G^0\rtimes\ZZ/2\ZZ$&\centering$u_0^2v_0+u_1u_2v_1=0$&$\PP^2\times\PP^1$
\\
\hline
\no\label{d=5:A1}
&$5$&$4$&$7$&$\type{A}_1$&$1$&\makecell{\ref{d=6:A1-3l},\\ \ref{d=6:A1-4l}}&$\Gm$&$G^0\times\DD_3$&\centering$u_0u_1v_0+(u_0+u_1)u_2v_1=0$&$\PP^2\times\PP^1$
\\
\hline
\no\label{d=6:A2-A1}
&$6$&$1$&$1$&$\type{A}_2\type{A}_1$&$6$&--&$\BB_3$&$G^0$&\centering---&$\PP(1,2,3)$
\\
\hline
\no\label{d=6:A2}
&$6$&$2$&$2$&$\type{A}_2$&$3$&--&$\UU_3\rtimes\Gm$&$G^0\rtimes\ZZ/2\ZZ$&\centering$y_1y_3=y_2^2+y_1'^4$&$\PP(1^2,2,3)$
\\
\hline
\no\label{d=6:2A1}
&$6$&$2$&$2$&$2\type{A}_1$&$2$&--&$\BB_2\times\BB_2$&$G^0$&\centering$y_1y_2=y_1'^2y_1''$&$\PP(1^3,2)$
\\
\hline
\no\label{d=6:A1-3l}
&$6$&$3$&$3$&$\type{A}_1$&$2$&--&$\Ga^2\rtimes\Gm$&$G^0\rtimes\DD_3$&\centering$y_1y_2=y_1'y_1''(y_1'+y_1'')$&$\PP(1^3,2)$
\\
\hline
\no\label{d=6:A1-4l}
&$6$&$3$&$4$&$\type{A}_1$&$1$&\ref{d=7}&$\BB_2\times\Gm$&$G^0\rtimes\ZZ/2\ZZ$&\centering$u_0v_0+u_1v_1+u_2v_2=0$, $u_0v_1+u_1v_2=0$&$\PP^2\times\PP^2$
\\
\hline
\no\label{d=6:smooth}
&$6$&$4$&$ 6$&--&$1$&\ref{d=7:smooth}&$\Gm^2$&$G^0\rtimes\DD_6$&\centering$u_0v_0w_0=u_1v_1w_1$&$\PP^1\times\PP^1\times\PP^1$
\\
\hline
\no\label{d=7}
&$7$&$2$&$2$&$\type{A}_1$&$1$&\ref{d=8}&$\BB_3$&$G^0$&&
\\
\hline
\no\label{d=7:smooth}
&$7$&$3$&$3$&--&$1$&\makecell{\ref{d=8:F1},\\ \ref{d=8:P1-P1}}&$\BB_2\times\BB_2$&$G^0\rtimes\ZZ/2\ZZ$&\centering & \\
\hline
\no\label{d=8}
&$8$&$1$&$0$&$\type{A}_1$&$4$&--&see Pr. \ref{c4}&$G^0$&\centering---&$\PP(1,1,2)$
\\
\hline
\no\label{d=8:F1}\centering
&$8$&$2$&$1$&--&$1$&\ref{d=9:P2}&see Pr. \ref{c4}&$G^0$&\centering $u_0v_0=u_1v_1$&$\PP^2\times\PP^1$
\\
\hline
\no\label{d=8:P1-P1}
&$8$&$2$&$0$&--&$2$&--&see Pr. \ref{c4}&$G^0\rtimes\ZZ/2\ZZ$&\centering---&$\PP^1\times\PP^1$
\\
\hline
\no\label{d=9:P2}
&$9$&$1$&$0$&--&$3$&--&see Pr. \ref{c4}&$G^0$&\centering---&$\PP^2$
\\
\hline
\end{longtable}
\end{center}
\end{landscape}
\pagestyle{plain}

\end{document}